\documentclass[a4paper,11pt,centertags,reqno,psamsfonts,draft]{amsart}

\usepackage{letltxmacro}        % http://ctan.org/pkg/letltxmacro
\usepackage{ifthen}		        % http://ctan.org/pkg/ifthen
\usepackage[headings]{fullpage} % http://ctan.org/pkg/fullpage
\usepackage{setspace}		    % http://ctan.org/pkg/setspace
\usepackage[inline]{enumitem}   % http://ctan.org/pkg/enumitem
\usepackage{array}				% http://ctan.org/pkg/array
\usepackage{fancybox}		    % http://ctan.org/pkg/fancybox     
\usepackage{graphicx}           % http://ctan.org/pkg/graphicx
\usepackage{subfigure}		    % http://ctan.org/pkg/subfigure
\usepackage{amsmath}	        % http://ctan.org/pkg/amsmath
\usepackage{amsthm}				% http://ctan.org/pkg/amsthm
\usepackage{amssymb}				
\usepackage{mathtools}			% http://ctan.org/pkg/mathtools
\usepackage{mathrsfs}			% http://ctan.org/pkg/mathrsfs
\usepackage{stmaryrd}			% http://ctan.org/pkg/stmaryrd
\usepackage{yhmath}				% http://ctan.org/pkg/yhmath
\usepackage{accents}			% http://ctan.org/pkg/accents
\usepackage{extarrows}          % http://ctan.org/pkg/extarrows
\usepackage{xfrac}              % http://ctan.org/pkg/xfrac
\usepackage{url}
\usepackage[all]{xy}			% http://ctan.org/pkg/xypic
\usepackage{comment}
\usepackage[round,comma,authoryear,sort&compress]{natbib} % http://ctan.org/pkg/natbib

\usepackage{enumitem}

\newcommand{\N}{\mathbb{N}}

\newcommand{\C}[2]{
\ifthenelse{#1=0 \and #2=0}{\textsf{\upshape C}}
{\ifthenelse{#2=0}{\textsf{\upshape C}^{#1}}
{\textsf{\upshape C}^{#1,#2}}}
}

\renewcommand{\d}{\mathrm{d}}

\newcommand{\E}{\textsf{\upshape E}}

\renewcommand{\P}{\textsf{\upshape P}}

\newcommand{\indicator}[1]{\mathbf{1}_{#1}}

\newcommand{\filt}[1]{\mathfrak{#1}}
\newcommand{\sigalg}[1]{\mathscr{#1}}

\newcommand{\lc}{[\![}
\newcommand{\rc}{]\!]}

\makeatletter
\let\oldr@@t\r@@t
\def\r@@t#1#2{%
\setbox0=\hbox{$\oldr@@t#1{#2\,}$}\dimen0=\ht0
\advance\dimen0-0.2\ht0
\setbox2=\hbox{\vrule height\ht0 depth -\dimen0}%
{\box0\lower0.4pt\box2}}
\LetLtxMacro{\oldsqrt}{\sqrt}
\renewcommand*{\sqrt}[2][\ ]{\oldsqrt[#1]{#2}}
\makeatother

\theoremstyle{plain}
\newtheorem{theorem}{Theorem}

\newtheorem{proposition}[theorem]{Proposition}
\newtheorem{corollary}[theorem]{Corollary}
\theoremstyle{definition}

\newtheorem{example}[theorem]{Example}

\theoremstyle{remark}

\numberwithin{equation}{section}
\numberwithin{figure}{section}
\numberwithin{table}{section}

\graphicspath{{Figures/}}

\bibpunct{(}{)}{,}{a}{}{,}

\allowdisplaybreaks[4]

\begin{document}
\title{Piecewise constant local martingales with bounded numbers of jumps}

\author{Johannes Ruf}

\address{Johannes Ruf\\
Department of Mathematics\\
London School of Economics and Political Science}

\email{j.ruf@lse.ac.uk}

\thanks{I thank Olga Becciv, Alekos Cecchin, Giorgia Callegaro, Martino Grasselli, and Wolfgang Runggaldier for helpful discussions on the subject matter of this note.
I am grateful for the generous support provided by the Oxford-Man Institute of Quantitative Finance at the University of Oxford.}



\date{\today}

\begin{abstract} 
A piecewise constant local martingale $M$ with boundedly many jumps is a uniformly integrable martingale if and only 
if $M_\infty^-$ is integrable. 
\end{abstract}

\maketitle

\section{Main theorem}
Let $(\Omega, \sigalg F, (\sigalg F_t)_{t \geq 0}, \P)$ denote a filtered probability space with $\bigcup_{t \geq 0} \sigalg F_t \subset \sigalg F$.
In Section~\ref{S:proof}, we shall prove the following theorem.
\begin{theorem}\label{T:finite1}
	Assume for some $N \in \N_0$ and some stopping times $0 \leq \rho_1 \leq \cdots \leq \rho_N$ we have a local martingale $M$ of the form
	\begin{align} \label{eq:160407.1}
		M = \sum_{m = 1}^N J_m \indicator{\lc \rho_m, \infty\lc}, \qquad \text{that is,}\qquad M_t = \sum_{m = 1}^N J_m \indicator{\{t \geq \rho_m\}}, \quad t \geq 0,
	\end{align}	
	where $J_m$ is $\sigalg F_{\rho_m}$--measurable for each $m = 1, \cdots, N$.  If
	\begin{align} \label{eq:160406.11}
		\E\left[\liminf_{t \uparrow \infty} M_t^-\right]  < \infty
	\end{align}
	then $M$ is a uniformly integrable martingale.
\end{theorem}

	In \eqref{eq:160406.11}, we could replace the limit inferior by a limit since $M$ only has finitely many jumps and hence converges to a random variable $M_\infty$. Hence, \eqref{eq:160406.11} is equivalent to $\E[ M_\infty^-]    < \infty$.

\begin{corollary} \label{C:finite1}
	Suppose the notation and assumptions of Theorem~\ref{T:finite1} hold, but with \eqref{eq:160406.11} replaced by 
	\begin{align*} 
		\E\left[M_t^-\right]  < \infty, \qquad t \geq 0.
	\end{align*}
	Then $M$ is a  martingale.
\end{corollary}
\begin{proof}
	Fix a deterministic time $T \geq 0$ and consider the local martingale $\widetilde M = M^T$; that is, $\widetilde M $ is the local martingale $M$ stopped at time $T$.   Then $\widetilde M$ satisfies the conditions of Theorem~\ref{T:finite1}, with $J_m$ replaced by $J_m \indicator{\{\rho_m \leq T\}}$ for each $m = 1, \cdots, N$.  Hence, $\widetilde M$ is a uniformly integrable martingale. Since $T$ was chosen arbitrarily the assertion follows.
\end{proof}

\citet{Jacod:Shiryaev:1998} prove the following special case of Theorem~\ref{T:finite1}.
\begin{proposition}  \label{P:JS98}
	Fix $N \in \N_0$ and assume we have a discrete-time filtration $\filt G = (\sigalg G_m)_{m = 0, 1, \cdots, N}$ and a $\filt G$--local martingale $Y = (Y_m)_{m = 0, 1, \cdots, N}$. If $\E[Y_N^-] < \infty$ then $Y$ is a $\filt G$--uniformly integrable martingale.
\end{proposition}

Note that Proposition~\ref{P:JS98} follows from Theorem~\ref{T:finite1}. Indeed, define the continuous-time process $M$ and the filtration $(\sigalg F_t)_{t \geq 0}$ by $M_t = Y_{[t] \wedge n}$  and $\sigalg F_t = \sigalg G_{[t] \wedge n}$, respectively, where $[t]$ denotes the largest integer smaller than or equal to $t$.  Then $M$ is a local martingale as in \eqref{eq:160407.1}, with $N$ replaced by $N+1$. To see this,  set $\rho_m = m-1$ and $J_m = Y_{m-1} - Y_{m-2}$ with $Y_{-1} := 0$, for each $m = 1, \cdots, N+1$. Applying Theorem~\ref{T:finite1} then yields Proposition~\ref{P:JS98}.

\section{Proofs of Theorem~\ref{T:finite1}} \label{S:proof}

In the following, we will provide two proofs of Theorem~\ref{T:finite1}. The first one assumes Proposition~\ref{P:JS98} is already shown and reduces the more general situation of Theorem~\ref{T:finite1}  to the discrete-time setup of Proposition~\ref{P:JS98}. The second proof does not assume Proposition~\ref{P:JS98}, but instead provides a direct argument based on an induction.

\begin{proof}[Proof~I, relying on Proposition~\ref{P:JS98}]  
Let us set $\rho = 0$ and $\rho_{N+1} = \infty$ and let $(\tau_n)_{n \in \N}$ denote a localization sequence of $M$ such that $M^{\tau_n}$ is a uniformly integrable martingale for each $n \in \N$.
For any stopping time $\tau$ we may define a sigma algebra
\begin{align*}
	\sigalg F_{\tau-} = \sigma\left( \left\{ A \cap \{t<\tau\},\, A \in \sigalg F_t,\, t \geq 0\right\} \cup \sigalg F_0\right).
\end{align*}
Note that $\{\tau = \infty\} = \bigcap_{n \in \N} \{n < \tau \} \in \sigalg F_{\tau-}$.

Let us now define a filtration $\filt G = (\sigalg G_m)_{m = 0, \cdots, N}$ and a process $Y = (Y_m)_{m = 0, 1, \cdots, N}$ by $\sigalg G_m = \sigalg F_{\rho_m} \vee\sigalg F_{\rho_{m+1}-}$ and $Y_m = M_{\rho_m}$, respectively.  Note that $Y$ is adapted to $\filt G$. Next, let us define a sequence $(\sigma_n)_{n \in \N}$ of random times, each taking values in $\{0, \cdots, N-1, \infty\}$ by
\begin{align*}
	\sigma_n = \sum_{m=0}^{N-1} m \indicator{\{\rho_m\leq \tau_n < \rho_{m+1} < \infty\}} + \infty \indicator{ \bigcup_{m = 0}^{N} \{\rho_m \leq \tau_n\} \cap  \{\rho_{m+1} = \infty\}}.
\end{align*}
Then, $\sigma_n$ is a $\filt G$--stopping time for each $n \in \N$ since
\begin{align*}
	\{\sigma_n = m\} &= 	\{\rho_m\leq \tau_n  < \rho_{m+1} < \infty\} \in \sigalg F_{\rho_m} \vee\sigalg F_{\rho_{m+1}-} = \sigalg G_m, \qquad m = 0, \cdots, N-1,
\end{align*}
and, furthermore, $\lim_{n \uparrow \infty} \sigma_n = \infty$.

We now fix $n \in \N$ and  prove that $Y^{\sigma_n}$ is a $\filt G$--martingale, which then yields that $Y$ is a $\filt G$--local martingale. To this end, we have, for each $m = 0, \cdots, N$,
\begin{align*}
	Y_m^{\sigma_n} &=  \sum_{k = 0}^{N-1} M_{\rho_{m \wedge k}} \indicator{\{\sigma_n = k\}} +  
			 M_{\rho_{m}} \indicator{\{\sigma_n = \infty\}} \\
			 &=  \sum_{k = 0}^{N-1} M_{\rho_{m \wedge k}}  \indicator{\{\rho_k\leq \tau_n  < \rho_{k+1} < \infty\}}  +  
			 M_{\rho_{m}}  \indicator{ \bigcup_{k = 0}^{N} \{\rho_k \leq \tau_n\} \cap  \{\rho_{k+1} = \infty\}}\\
			 &= M^{\tau_n}_{\rho_m},
\end{align*}
yielding
$
	\E[|Y_m^{\sigma_n}|]  < \infty.
$
Now, fix $m= 1, \cdots, N$. First, for any $A \in \sigalg F_{\rho_m-1}$, we have
\begin{align*}
	\E[Y_m^{\sigma_n} \indicator{A}]  = \E[M^{\tau_n}_{\rho_m} \indicator{A}] = \E[M^{\tau_n}_{\rho_{m-1}}\indicator{A}] = \E[Y_{m-1}^{\sigma_n} \indicator{A}];
\end{align*}
next, for any $t \geq 0$ and $A \in \sigalg F_t$, we have
\begin{align*}
	\E[Y_m^{\sigma_n} \indicator{A \cap \{t < \rho_m\}}]  &= \E[M^{\tau_n}_{\rho_m} \indicator{A \cap \{t < \rho_m\}}] = \E[M^{\tau_n}_{t} \indicator{A \cap \{t < \rho_m\}}]  = \E[M^{\tau_n}_{\rho_m-1} \indicator{A \cap \{t < \rho_m\}}] \\
	&= \E[Y_{m-1}^{\sigma_n} \indicator{A \cap \{t < \rho_m\}}], 
\end{align*}
yielding that $\E[Y_m^{\sigma_n} \indicator{A}]  = \E[Y_{m-1}^{\sigma_n} \indicator{A}]$ for all $A \in \sigalg G_{m-1}$. Hence, $Y$ is indeed a $\filt G$--local martingale.

The assumptions of the theorem yield that $\E[Y^-_N] < \infty$; hence $Y$ a $\filt G$--uniformly integrable martingale by Proposition~\ref{P:JS98}. Now, fix $t \geq 0$ and $A \in \sigalg F_t$. Then we get $\E[|M_t|] + \E[|M_\infty|] \leq 2 \sum_{m=0}^N \E[|N_m|] < \infty$ and  
\begin{align*}
	\E[M_\infty \indicator{A}] &= \sum_{m = 0}^N \E[Y_N \indicator{A \cap \{\rho_m \leq t < \rho_{m+1}\}}]
		= \sum_{m = 0}^N \E[Y_m \indicator{A \cap \{\rho_m \leq t < \rho_{m+1}\}}]\\
		&= \sum_{m = 0}^N \E[M_t \indicator{A \cap \{\rho_m \leq t < \rho_{m+1}\}}] = \E[M_t \indicator{A}]
\end{align*}
since $A \cap \{\rho_m \leq t < \rho_{m+1}\} \in \sigalg G_m$ for each $m = 0, \cdots, N$. Hence, $M$ is indeed a uniformly integrable martingale.
\end{proof}

\begin{proof}[Proof~II, relying on an induction argument]
	We proceed by induction over $N$. The case $N=0$ is clear.  Hence, let us assume the assertion is proven for some $N \in \N_0$ and consider the assertion with $N$ replaced by $N+1$.  Let $(\tau_n)_{n \in \N}$ denote a corresponding localization sequence such that $M^{\tau_n}$ is a uniformly integrable martingale for each $n \in \N$.
		
	\emph{Step~1}: In the first step, we want to argue that 
	the nondecreasing sequence $(\widehat \tau_n)_{n \in \N}$, given by
	\begin{align*}
		\widehat \tau_n = \tau_n \indicator{\{\tau_n < \rho_1\}} + \infty  \indicator{\{\tau_n \geq \rho_1\}} \geq \tau_n, 
	\end{align*}
	is also a localization sequence for $M$.  To this end, fix $k \in \N$ and consider the process
	\begin{align*} 		
		\widetilde M = (M - M^{\tau_k}) \indicator{\{\tau_k \geq \rho_1\}}.
	\end{align*}
	Then we have
	\begin{align*}
		\widetilde M^-   \leq M^- +  |M^{\tau_k }| ;
	\end{align*}
	hence
	\begin{align} \label{eq:160406.1}
		\E\left[\liminf_{t \uparrow \infty} \widetilde M_t^-\right]  \leq   \E\left[\liminf_{t \uparrow \infty} M_t^-\right] + \E\left[|M_\infty^{\tau_k} | \right] < \infty.
	\end{align}
	Next, we argue that $\widetilde M$ is also a local martingale, again with localization sequence 
	$(\tau_n)_{n \in \N}$. Indeed,  for $n \in \N$, $t,h \geq 0$, and $A \in \sigalg F_t$  note that 
	\begin{align*}
		\E\left[\widetilde{M}^{\tau_n}_{t+h} \indicator{A}\right] 
			&= \E\left[\left({M}^{\tau_n}_{t+h} - {M}^{\tau_n \wedge \tau_k}_{t+h} \right)\indicator{A \cap \{\rho_1 \leq \tau_k \leq t\}}\right]   +  \E\left[\left({M}^{\tau_n}_{t+h} - {M}^{\tau_n \wedge \tau_k}_{t+h}  \right) \indicator{A \cap \{\rho_1 \leq \tau_k\} \cap \{\tau_k > t\}}\right] \\
			&=  \E\left[\left({M}^{\tau_n}_{t} - {M}^{\tau_n \wedge \tau_k}_{t}  \right)\indicator{A \cap \{\rho_1 \leq \tau_k \leq t\}}\right]   +  \E\left[\left( {M}^{\tau_n \wedge \tau_k}_{t+h} -  {M}^{\tau_n \wedge \tau_k}_{t+h} \right)\indicator{A \cap \{\rho_1 \leq \tau_k\} \cap \{\tau_k > t\}}\right] 
			\\
			&=  \E\left[\widetilde {M}^{\tau_n}_{t}  \indicator{A}\right],
	\end{align*}
	where we used the definition of $\widetilde M$, $\{\rho_1 \leq \tau_k \leq t\} \in \sigalg F_t$,  $A \cap \{\rho_1 \leq \tau_k\} \cap \{\tau_k > t\} \in \sigalg F_{\tau_k}$, and the martingale property of $M^{\tau_n}$.  Alternatively, we could have observed that $\widetilde M_\cdot = \int_0^\cdot \indicator{\{\rho_1 \leq \tau_k < s\}} \d M_s$ (using the fact that $\indicator{\{\rho_1 \leq \tau_k\}} \indicator{\rc\tau_k, \infty\lc\}}$ is bounded and predictable since it is adapted and left-continuous).  Hence, $\widetilde M$ is a local martingale of the form
	\begin{align*}
		\widetilde M = \sum_{m = 2}^{N+1}  \left(J_m \indicator{\{\rho_1 \leq \tau_k < \rho_m\}} \right)\indicator{\lc \rho_m, \infty\lc},
	\end{align*}	
	satisfying \eqref{eq:160406.1}, and the induction hypothesis yields that $\widetilde M$ is a uniformly integrable martingale. This again yields that 
	\begin{align*}
		M^{\widehat \tau_k} = M^{\tau_k} + \widetilde M
	\end{align*}
	is also a uniformly integrable martingale, proving the claim that  $(\widehat \tau_n)_{n \in \N}$ is a localization sequence for $M$.
	
	\emph{Step~2}: We want to argue that   $M_t \in \mathcal L^1$ for each $t \in [0,\infty]$. To this end, fix $t \in [0,\infty]$ and note
	\begin{align}
		\E[|M_t|] &\leq \liminf_{n \uparrow \infty} \E\left[\left|M_t^{\widehat \tau_n}\right|\right]  \label{eq:161221.1} \\
			&= 
			\E[M_0] + 2 \liminf_{n \uparrow \infty} \E\left[\left({M_t^{\widehat \tau_n}}\right)^- \right] \label{eq:161221.2} \\
			&\leq
			\E[M_0] +  2 \liminf_{n \uparrow \infty} \E\left[\left(M_\infty^{\widehat \tau_n}\right)^- \right] \label{eq:161221.3} \\
			&\leq \E[M_0] + 2 \E[M_\infty^- ] \label{eq:161221.4} \\
			&< \infty. \label{eq:161221.5}
	\end{align}	
	Here, the inequality in \eqref{eq:161221.1}  is an application of Fatou's lemma.   The equality in  \eqref{eq:161221.2}  relies on the fact that for any uniformly integrable martingale $X$  we have $\E[|X_t|] = \E[X_t^+] + \E[X_t^-] = \E[X_0] + 2 \E[X_t^-]$. The  inequality in  \eqref{eq:161221.3}  uses that $(M^{\widehat \tau_n})^-$ is a uniformly integrable submartingale, thanks to Jensen's inequality, for each $n \in \N$. The inequality   in  \eqref{eq:161221.4}  (which is, actually, an equality) uses the fact that $M_{\widehat \tau_n} \in \{0, M_\infty\}$, for each $n \in \N$, by construction of the localization sequence $(\widehat \tau_n)_{n \in \N}$. Finally, the inequality  in  \eqref{eq:161221.5} holds by assumption.
	
	\emph{Step~3}: We now argue that  $M$ is a uniformly integrable martingale. To this end, fix $t\geq 0$ and $A \in \sigalg F_t$. Observe that
	\begin{align}
		\E\left[M_{\infty} \indicator{A}\right]
		&= \lim_{n \uparrow \infty} \left( \E\left[ M_{\infty} \indicator{A \cap \{\widehat \tau_n < \rho_1 < \infty\}}\right]  +  \E\left[ M_{\infty} \indicator{A \cap \{\widehat \tau_n < \rho_1\} \cap \{\rho_1 =  \infty\}}\right] + \E\left[ M_{\infty} \indicator{A \cap \{\widehat \tau_n \geq \rho_1\}}\right] \right)  \nonumber\\
		&= \lim_{n \uparrow \infty} \E\left[M_{\infty}^{\widehat{\tau}_n} \indicator{A \cap \{\widehat \tau_n = \infty\}}\right] \label{eq:160406.2}\\
		&= \lim_{n \uparrow \infty} \left(\E\left[M_{\infty}^{\widehat{\tau}_n} \indicator{A \cap \{\widehat \tau_n > t\}}\right]  - \E\left[M_{\infty}^{\widehat{\tau}_n} \indicator{A \cap \{ t <\widehat \tau_n < \infty\}}\right] \right) \nonumber \\
		&=  \lim_{n \uparrow \infty} \E\left[M_{t}^{\widehat{\tau}_n} \indicator{A \cap \{\widehat \tau_n > t\}}\right]  \label{eq:160406.3}\\ 
		&= 		\E\left[M_{t} \indicator{A}\right].   \label{eq:160406.4}
	\end{align}
	We obtained the equality in \eqref{eq:160406.2} since $\widehat \tau_n = \infty$ on the event $\{\widehat \tau_n \geq \rho_1\}$, and since the first  term on the left-hand side is zero by the dominated convergence theorem and  the second one thanks to the form of $M$.  
	 In \eqref{eq:160406.3}, we used the martingale property of $M^{\widehat \tau_n}$ in the first term and the fact that $M_{\widehat {\tau}_n} = 0$ on the  event $\{\widehat\tau_n <\infty\}$ in the second term, for each $n \in \N$.  Finally, we exchanged limit and expectation in \eqref{eq:160406.4}  again by an application of the dominated convergence theorem.  This then concludes the proof.
\end{proof}

\section{Two examples concerning the assumptions  in Theorem~\ref{T:finite1} }

\begin{example}
	Assume $(\Omega, \sigalg F, \P)$ allows for a sequence $(\theta_m)_{m \in \N}$ of independent random variables with $\P[\theta_1=2]=1$ and $\P[\theta_m = -1] = \sfrac{1}{2} =  \P[\theta_m = 1]$ for all $m \geq 2$.  Fix families $(J_m)_{m \in \N}$ and $(\rho_m)_{m \in \N}$ of random variables with 
	\begin{align*}
		J_m = 2^{m-2} \theta_{m} \qquad \text{and} \qquad \rho_m = (1-\sfrac{1}{m}) \indicator{\bigcap_{k=2}^{m-1} \{\theta_k = 1\}}  + \infty \indicator{\bigcup_{k=2}^{m-1} \{\theta_k = -1\}}.
	\end{align*}
	Next, define $M$ as in \eqref{eq:160407.1} with $N = \infty$ and assume that $(\sigalg F_t)_{t \geq 0}$ is the filtration generated by $M$.  Then $M$ is a local martingale, with localization sequence $(\rho_m)_{m \in \N}$.  Indeed, $M$ is a process that starts in one, and then, at times $\sfrac{1}{2}, \sfrac{2}{3}, \cdots$ doubles its value  or jumps to zero, each with probability $\sfrac{1}{2}$.  Since it eventually jumps to zero as $\P[\bigcup_{m = 2}^\infty \{\theta_m = -1\}] = 1$, we have $M_1= 0$.   In particular, $M$ is not a true martingale, but satisfies $\E[M_1^-] = 0 < \infty$.  Thus, the assertions of Theorem~\ref{T:finite1} or Corollary~\ref{C:finite1} are not valid if $N = \infty$, even if  $\P[\bigcup_{m \in \N} \{\rho_m=\infty\}] = 1$.  \qed
\end{example}

The next example illustrates that the assumptions of Corollary~\ref{C:finite1} are not sufficient to guarantee that $M$ is a uniformy integrable martingale, even if there is only one jump possible, that is, even if $N=1$.  The example is adapted from \cite{Ruf_Cherny}, where it is used as a counterexample for a different conjecture.

\begin{example}  \label{Ex:7}
Let $\rho$ be an $\N \cup \{\infty\}$--valued random variable with 
\begin{align*}
	\P\left[\rho = i\right] = \frac{1}{2 i^2}, \qquad i \in \N.
\end{align*}
This then yields that
\begin{align*}
\P\left[\rho = \infty\right] = 1 - \frac{\pi^2}{12}.
\end{align*}
Moreover, let $\theta$ be an independent $\{-1,1\}$ valued random variable with $\P[\theta = 1] = \P[\theta = -1] = \sfrac{1}{2}$. Define $J = \theta \rho^2$. Then the stochastic process
\begin{align*}
	M = J \indicator{\lc \rho, \infty\lc},
\end{align*}
along with the filtration $(\sigalg F_t)_{t \geq 0}$ it generates, satisfies exactly the conditions of Corollary~\ref{C:finite1}.  Indeed, $\rho$ is an $\filt F$--stopping time and $M^-_t \leq \rho^2 \indicator{\{\rho \leq t\}} \leq t^2$, hence $M^-_t \in \mathcal L^1$ for each $t \geq 0$.  Thus, $M$ is a martingale. This fact would also be very easy to check by hand.

We have $M_\infty = \lim_{t \uparrow \infty} M_t$ exists and satisfies $|M_\infty| = \rho^2 \indicator{\{\rho<\infty\}}$. 
Thus,
\[
	\E[|M_\infty|] =  \sum_{i \in \N} i^2 \frac{1}{2 i^2} = \infty,
\]
and $M$ cannot be a uniformly integrable martingale. 
 \qed
\end{example}

% The references.
\bibliography{aa_bib}

\begin{thebibliography}{}

\bibitem[\protect\citeauthoryear{Jacod and Shiryaev}{Jacod and
  Shiryaev}{1998}]{Jacod:Shiryaev:1998}
Jacod, J. and A.~N. Shiryaev (1998).
\newblock Local martingales and the fundamental asset pricing theorems in the
  discrete-time case.
\newblock {\em Finance Stoch.\/}~{\em 2\/}(3), 259--273.

\bibitem[\protect\citeauthoryear{Ruf}{Ruf}{2015}]{Ruf_Cherny}
Ruf, J. (2015).
\newblock The uniform integrability of martingales. {O}n a question by
  {A}lexander {C}herny.
\newblock {\em Stochastic Processes and their Applications\/}~{\em 125\/}(10),
  3657--3662.

\end{thebibliography}
\bibliographystyle{chicago}
\end{document}